\documentclass[12pt]{article}

\usepackage{a4wide}
\usepackage{amsmath}
\usepackage{amssymb}
\usepackage{amsthm}
\DeclareMathOperator{\cor}{cor}

\theoremstyle{plain}
\newtheorem{theorem}{Theorem}
\newtheorem{lemma}[theorem]{Lemma}
\newtheorem{corollary}[theorem]{Corollary}

\newcommand{\setmid}{\,|\,}
\newcommand{\N}{{\mathbb{N}}}

\newcommand{\R}{{\mathbb{R}}}

\newcommand{\lf}{\left\lfloor}
\newcommand{\rf}{\right\rfloor}
\newcommand{\lc}{\left\lceil}
\newcommand{\rc}{\right\rceil}
\newcommand{\la}{\left|}
\newcommand{\ra}{\right|}

\newcommand{\mmod}[1]{\,\,(\mbox{mod}\,\,#1)}

\begin{document}

\title{Winkler's Hat Guessing Game: Better Results for Imbalanced Hat Distributions\thanks{This is an old result of mine never published except in my Habilitation thesis.}}

\author{Benjamin Doerr\\ Max-Planck-Institute for Informatics\\ 66123 Saarbr\"ucken\\ Germany}

\maketitle

\begin{abstract}
  In this note, we give an explicit polynomial-time executable strategy for Peter Winkler's hat guessing game that gives superior results if the distribution of hats is imbalanced. While Winkler's strategy guarantees in any case that $\lfloor n/2 \rfloor$ of the $n$ player guess their hat color correct, our strategy ensures that the players produce $\max\{r,b\} - 1.2 n^{2/3} -2$ correct guesses for any distribution of $r$ red and $b = n - r$ blue hats. We also show that any strategy ensuring $\max\{r,b\} - f(n)$ correct guesses necessarily has $f(n) = \Omega(\sqrt n)$.  
\end{abstract}

\section{The Hat Color Guessing Game}

In this note, we deal with the following game suggested by Peter Winkler~\cite{winklergames}. In the \emph{simultaneous hat guessing game}, there are $n$~players each wearing a red or blue hat. Each player can see all hats
  except his own. Simultaneously, the players have to guess the color of their
  own hat.  No communication is allowed during the game. The players may,
  however, discuss their strategy before they get to see the hats.

It is easy to see from the rules that no player can make sure that he
guesses his hat color, no matter what strategy the players agree on.
Thus the following result in~\cite{winklergames} is quite surprising.
There is a strategy that \emph{guarantees} that $\lf n/2 \rf$ players
guess their hat color correctly. In fact, the strategy is not too
difficult, and the interested reader is encouraged to stop reading now
and try to find such a strategy on his own.

A drawback of this strategy is
that is also ensures that not more than $\lc n/2 \rc$ players guess correctly.
This is particularly annoying since there is a strategy (and it is probably the
first one most people think of) that seems to gain an advantage from the fact
that there are more hats in one color than the other.

Assume for simplicity that $n$ is even and that there are $r$ red hats and $b = n-r$ blue ones. The \emph{majority strategy}\/ is
for each player to guess that color which he can see more hats in. If there are
more red hats than blue ones, the assumption that $n$ is even ensures that the
difference is at least $2$. Thus all players can see more red hats than blue
ones. Using the majority strategy, all players wearing a red hat guess right.
Hence this strategy is superior, leading to $\max\{r,b\}$ correct guesses, if the distribution of hats is imbalanced.
Unfortunately, the majority strategy fails badly if there are as many red as
blue hats. In this case, all player can see more hats in the color they are not
wearing. Hence \emph{all}\/ players guess wrong.

In this note, we are looking for strategies that combine advantages  of
the 50\%--strategy and the majority strategy. We present an explicit strategy that
produces more correct guesses if the distribution of hats is imbalanced, but
ensures that at least nearly 50\% of the players guess right in any case. A
probabilistic argument shows that no strategy can ensure the better outcome of
the majority and the 50\%--strategy in all cases. Thus to exploit imbalanced
distributions, one has to pay a price in the sense that less than half of the
players guess right for balanced partitions. But this price can be kept small:
Our strategy  produces $\max\{b,r\} - o(n)$ correct guesses on any
distribution of $r$ red and $b$ blue hats. More precisely, we show the following.

{\sloppy \begin{theorem}\label{thm}
  There is an explicit  strategy such that $n$ players surely produce $\max\{r,b\} - 1.2 n^{2/3} -2$ correct guesses for any distribution of $r$ red and $b = n - r$ blue hats. This strategy requires the players to do only elementary, polynomial-time computations.
\end{theorem}

We also show that no strategy can provide a guarantee of better than
\mbox{$\max\{r,b\} - \Omega(\sqrt{n})$}.

Subsequent to the first version of this work, Uriel Feige~\cite{feige} gave an existential proof for a strategy producing $\max\{r,b\} - O(\sqrt n)$ correct guesses, but left open the problem whether there is a strategy of this quality such that all computations done by the players can be performed in polynomial time.

}

\section{Notation and the Pairing Strategy}

Let us assume from now on that $n$ is even unless otherwise stated. We shall
show that in this case there is a strategy ensuring $\max\{r,b\} - 1.2 n^{2/3}
-1$ correct guesses for any distribution of $r$ red and $b$ blue hats. This
yields Theorem~\ref{thm}.

\subsection{Notation}

Let the set of players simply be~$[n]$. Then a distribution of hats is an
$\omega$ from $\Omega := \{R, B\}^n$. Put $R_\omega := \{i \in [n] \setmid
\omega_i = R\}$ and $B_\omega := [n] \setminus R_\omega$. Formally, a
\emph{strategy} for the players is a function $S : \Omega \to \Omega$ such that
for all $i \in [n]$ the $i$-th player's guess $S(\omega)_i$ is independent of
$\omega_i$. For a strategy $S$ and $\omega \in \Omega$ put \mbox{$\cor(S, \omega) =
  |\{i \in [n] \setmid S(\omega)_i = \omega_i\}|$}, the number of correct
guesses produced by $S$ on $\omega$.

\subsection{Pairing Strategy}

We brief{}ly review from~\cite{winklergames} the strategy ensuring $\tfrac n 2$
correct guesses. Assume the set $[n]$ of players partitioned into ordered pairs,
i.e., there are $x_i, y_i \in [n]$, $i \in [n/2]$, such that $\{x_i, y_i \setmid
i \in [n/2]\} = [n]$. Assume further that this pairing is known to the players.
The \emph{pairing strategy}\/ with respect to this pairing is as follows: For all
$i \in [n/2]$, player $x_i$ calls the color of $y_i$'s hat, player $y_i$ calls
the opposite color of $x_i$'s hat.  Thus if $x_i$'s and $y_i$'s hat have the
same color, then $x_i$ guesses right and $y_i$ wrong. If their hat colors are
different, $y_i$'s guess is right and $x_i$'s is wrong. In particular, this
strategy ensures that exactly one player from each pair guesses right.

Note that the pairing is independent of the hat colors. Thus the player may
agree on the pairing prior to the guessing as part of their agreement on a
strategy. Doing so and playing the pairing strategy ensures $\tfrac n2$ correct
guesses.

\section{Probabilistic Analysis and Lower Bounds}

As indicated in the introduction, the pairing strategy not only guarantees that
$\tfrac n2$ players guess right, it also guarantees that that many players guess
wrong. This cannot be helped as can be seen from the following elementary probabilistic argument, which was already sketched in~\cite{winklergames}.

Assume that we pick a distribution of hats uniformly at random from $\Omega$,
i.e., we view $\Omega$ as a probability space with probability distribution $\Pr :
\Omega \to [0,1]$ defined by $\Pr(\omega) = \tfrac 1 {|\Omega|} = 2^{-n}$ for all
$\omega \in \Omega$. Then any strategy in expectation produces $\tfrac n2$
correct guesses.

\begin{lemma}\label{lavg}
  Let $S : \Omega \to \Omega$ be any strategy. Then the expected number of correct
  guesses produced by $S$ on a random hat distribution is $\tfrac n2$.
\end{lemma}

\begin{proof}
  Define the following random variables. Denote by $X$ the number of correct
  guesses produced by $S$. For $i \in [n]$ let $X_i$ be $1$, if player $i$
  guesses correctly, and $0$ otherwise. Then $X = \sum_{i = 1}^n X_i$ and $\Pr(X_i
  = 1) = \Pr(X_i = 0) = \tfrac 12$. Thus $EX = \sum_{i = 1}^n EX_i = \tfrac n2$.
\end{proof}

Thus from the view-point of average case analysis, the game regarded is rather
boring. Lemma~\ref{lavg} has a nice combinatorial corollary.

\begin{corollary}\label{corcomb}
  For all even $n \in \N$,
  \[\sum_{{0 \le i \le n} \atop {i \neq n/2}} \binom{n}{i}
  \max\{i,n-i\} = 2^n \tfrac n 2.\]
\end{corollary}

\begin{proof}
  The expected number of correct guesses produced by the maximum
  strategy just is $2^{-n} \sum_{{0 \le i \le n} \atop {i \neq n/2}}
  \binom{n}{i} \max\{i,n-i\}$. Hence the claim follows from Lemma~\ref{lavg}.
\end{proof}

Another consequence of the lemma is that no strategy can ensure
$\max\{|R_\omega|,|B_\omega|\}$ correct guesses for all $\omega \in
\Omega$. More precisely, we obtain the following.

{\sloppy\begin{lemma}
  For any $n$, there is no strategy that produces more than
  \mbox{$\max\{|R_\omega|,|B_\omega|\} - \sqrt{n/(2\pi)} \exp(-1/(3n)) + 1$} correct
  guesses on all $\omega \in \Omega$.
\end{lemma}

\begin{proof} Let first $n$ be even. Assume there is a strategy $S$ such
    that $\cor(S,\omega) \ge \max\{|R_\omega|,|B_\omega|\} - \sqrt{n/(2\pi)}
    \exp(-1/(3n))$ for all $\omega \in \Omega$. From Lemma~\ref{lavg}
    and Corollary~\ref{corcomb} we have
  \begin{eqnarray*}
    2^n \tfrac n 2 &=& \sum_{\omega \in \Omega} \cor(S,\omega) \\
    &>& \sum_{\omega \in \Omega} \left(\max\{|R_\omega|,|B_\omega|\} -
    \sqrt{n/(2\pi)} \exp(-1/(3n))\right) \\
    &=& \sum_{i = 0}^n \binom{n}{i} \max\{i,n-i\} - 2^n \sqrt{n/(2\pi)}
    \exp(-1/(3n)) \\
    &=& 2^n \tfrac n 2 + \binom{n}{n/2} \tfrac n2 - 2^n \sqrt{n/(2\pi)}
    \exp(-1/(3n)).
  \end{eqnarray*}
  Estimating \mbox{$\binom{n}{n/2} \ge 2^n \sqrt{2/(\pi n)}
    \exp(-1/(3n))$}, cf.\ e.g. Robbins~\cite{rob}, yields a
    contradiction.
    
    Now let $n$ be odd and $S$ any strategy. Extend $S$ to a strategy
    $S'$ for $n+1$ players by letting the \mbox{$(n+1)$-st} player
    always guess $R$ and all other players ignore the
    \mbox{$(n+1)$-st} player's hat. By the above, there is an $\omega'
    \in \{R,B\}^{n+1}$ such that $\cor(S',\omega') \le
    \max\{|R_{\omega'}|,|B_{\omega'}|\} - \sqrt{(n+1)/(2\pi)}
    \exp(-1/(3n+3))$.  For $\omega = (\omega_1, \ldots, \omega_n)$ we
    have
  \begin{eqnarray*}
    \cor(S,\omega) &\le& \cor(S',\omega') \\
    &\le& \max\{|R_{\omega'}|,|B_{\omega'}|\} - \sqrt{(n+1)/(2\pi)}
    \exp(-1/(3n+3)) \\
    &\le& \max\{|R_{\omega}|,|B_{\omega}|\} +1 - \sqrt{n/(2\pi)}
    \exp(-1/(3n)).
  \end{eqnarray*}
\end{proof}}

\section{An Explicit Strategy}

It seems that an easy solution to our problem is to play the pairing strategy,
if the distribution of hats is balanced, and the majority strategy otherwise. It
turns out that this does not work. The problem is that the pairing strategy
works well only if both players from each pair apply is. Thus we needed to
ensure that either all players apply the majority strategy or all apply the
pairing strategy. The problem with our initial idea is that depending on whether
he is wearing a majority color hat or not, the player regards the distribution
as more or less balanced. Thus is seems difficult do get the players organized.
We solve this problem as follows.

\subsection{Partial Strategies $S(T,a,b)$}

For our strategy we also assume that the players have agreed on a pairing as
above. We say that a subset $T \subseteq [n]$ of the players respects the
pairing, if no pair intersects $T$ non-trivially, i.e., if $\{x_i, y_i\}
\subseteq T$ or $\{x_i, y_i\} \cap T = \emptyset$ holds for all $i \in [n/2]$.

For such a subset $T$ and integers $a, b$ such that $a < \tfrac 12 |T|
\le b$ and $a + 2 \le b$, we define the following strategy $S(T,a,b)$ for
the players in $T$: If a player in $T$ can see at least $b$ red hats
in $T$, he guesses `red'. If he can see at most $a$ red hats in $T$,
he guesses `blue'.  Otherwise he guesses according to the pairing
strategy.

\begin{lemma}\label{leins}
  Let $\omega  \in \Omega$ be any distribution of hats
  and $m = \max\{|B_\omega \cap T|, |R_\omega \cap T|\}$. Then the strategy
  $S(T,a,b)$ produces at least $\cor(\omega,T,a,b)$ correct guesses, where
  \[\cor(\omega,T,a,b) := \left\{\begin{array}{ll}
      m  = |R_\omega \cap T|& \mbox{if } |R_\omega \cap T| > b, \\
      b - \tfrac 12 |T| = m - \tfrac 12 |T| & \mbox{if } |R_\omega \cap T| = b, \\
      \tfrac 12 |T| & \mbox{if } a+2 \le |R_\omega \cap T| \le b-1,\\
      \tfrac 12 |T| -a -1 =  m - \tfrac 12 |T| & \mbox{if } |R_\omega \cap T| = a+1,\\
      m = |B_\omega \cap T|& \mbox{if } |R_\omega \cap T| \le a.\end{array}
  \right.\]
\end{lemma}

\begin{proof}
  If $|R_\omega \cap T| > b$, then all players can see at least $b$ red hats.
  Thus they all guess `red' and $m = |R_\omega \cap T|$ of them naturally are
  right.  If $|R_\omega \cap T| = b$, then those players wearing a blue hat can
  see $b$ red ones and (wrongly) guess `red', whereas the players wearing a red
  hat guess according to the pairing strategy. Since there are only $|T| - b$
  blue hats, at least $2b - |T|$ players wearing a red hat have a partner
  wearing a red hat as well.  Hence from at least $b - \tfrac12 |T|$ pairs
  both partners  guess according to the pairing strategy, producing one correct
  guess (and one false one) per pair.

  If $|R_\omega \cap T| \in \{a+2, \ldots, b-1\}$, then all players guess
  according to the pairing strategy, which yields $\tfrac 12 |T|$ correct
  guesses. If $|R_\omega \cap T| = a+1$, the players wearing a red hat can see
  only $a$ red hats and thus (wrongly) guess `blue'.  The $m$ players wearing a
  blue hat can see $a+1$ red hats and hence guess according to the pair
  strategy. As above, this produces at least $\tfrac 12 |T| -a -1$ correct
  guesses. If $|R_\omega \cap T| \le a$, all players guess `blue', $m =
  |B_\omega \cap T|$ of them being correct.
\end{proof}

\subsection{A Strategy for all Players}

The strategy $S(T,a,b)$ is not bad unless there are exactly $a+1$ or
$b$ red hats in $T$. In this case we say that $S(T,a,b)$ \emph{fails}.
Our plan is to partition the set of all players $[n]$ into $k \ge 2$
subsets $T_1, \ldots, T_k$ respecting the pairing and choose integers
$a_i, b_i$ for all $i \in [k]$ in such a way that at most one strategy
$S(T_i,a_i,b_i)$ fails.

Assume the partition $[n] = T_1 \dot\cup \ldots \dot\cup T_k$ be given and known
to the players. For all $i \in [k]$ put $\overline T_i = [n] \setminus T_i$. Let
$b_i$ be minimal subject to $b_i \ge \tfrac 12 |T_i|$ and
\[\la R_\omega \cap \overline T_i\ra + b_i \equiv i \mmod k.\]
Note that each player in $T_i$ can compute this number as he only needs to know
the number of red hats in $[n] \setminus T_i$. Put $a_i = b_i - k-1$.

\begin{lemma}\label{lfail}
  Let $\omega  \in \Omega$ be any distribution of hats and $T_i, a_i,
  b_i$ as above. Let $i \in [k]$ such that $i \equiv |R_\omega| \mmod k$. Then no
  strategy $S(T_j, a_j, b_j)$, $j \neq i$, fails.  $S(T_i,a_i,b_i)$ fails if and
  only if $|R_\omega \cap T_i| \in \{a_i+1, b_i\}$.
\end{lemma}

\begin{proof}
  Let $j \in [k]$ such that $S(T_j,a_j,b_j)$ fails. Then $|R_\omega \cap T_j|
  \in \{a_j+1,b_j\}$. Note that $a_j + 1 \equiv b_j \mmod k$ by definition. Hence
  \[|R_\omega| = \la R_\omega \cap \overline T_j \ra + |R_\omega \cap T_j| \equiv
  \la R_\omega \cap \overline T_j \ra + b_j \equiv j \mmod k.\] Thus at most one
  strategy may fail, namely the strategy $S(T_i,a_i,b_i)$. This happens if and
  only if \mbox{$|R_\omega \cap T_i| \in \{a_i+1, b_i\}$}.
\end{proof}

Let $S$ be the union of the strategies $S(T_i,a_i,b_i)$, $i \in [k]$, i.e., the
strategy such that a player contained in $T_i$ follows the strategy
$S(T_i,a_i,b_i)$.

{\sloppy\begin{lemma} For all $\omega \in \Omega$, \[\cor(S,\omega) \ge
    \max\{|R_\omega|,|B_\omega|\} - \tfrac 12 \max_{i \in [k]} |T_i| - (k-1)^2.\]
\end{lemma}}

\begin{proof}
  {\sloppy Let $i \in [k]$ such that $|R_\omega| \equiv i \mmod k$.
    Assume that \mbox{$\max\{|R_\omega|,|B_\omega|\} = |R_\omega|$}.
    Let $j \neq i$. Then $S(T_j,a_j,b_j)$ does not fail by
    Lemma~\ref{lfail}. From Lemma~\ref{leins} we conclude that if
    $|R_\omega \cap T_j| \in \{a_j+2, \ldots, b_j -1\}$, then
    \begin{eqnarray*}
      \cor(\omega,T_j,a_j,b_j) 
      &=& \tfrac 12 |T_j| = |R_\omega \cap T_j| - (|R_\omega \cap  T_j| - \tfrac 12 |T_j|)\\ 
      &\ge& |R_\omega \cap T_j| - (b_j - 1 - \tfrac 12 |T_j|) \ge |R_\omega 
      \cap T_j| - (k-2).
    \end{eqnarray*}
    If $|R_\omega \cap T_j| > b_j$, then $\cor(\omega,T_j,a_j,b_j) =
    |R_\omega \cap T_j|$, and if $|R_\omega \cap T_j| \le a_j$, then
    $\cor(\omega,T_j,a_j,b_j) = |B_\omega \cap T_j| \ge |R_\omega \cap
    T_j|$. Hence in all cases we have $\cor(\omega,T_j,a_j,b_j)\ge
    |R_\omega \cap T_j| - (k-2)$. The possibly failing strategy
    $S(T_i,a_i,b_i)$ yields $\cor(\omega,T_i,a_i,b_i) = \max\{|B_\omega
    \cap T_i|, |R_\omega \cap T_i|\} - \tfrac 12 |T_i| \ge |R_\omega
    \cap T_i| - \tfrac 12 |T_i|$ correct guesses.

  Thus the total number of correct guesses is
  \begin{eqnarray*}
    \cor(S,\omega) &\ge& |R_\omega \cap T_i| - \tfrac 12 |T_i| + \sum_{j \in [k] \setminus
      \{i\}} \left(|R_\omega \cap T_j| - (k-2)\right) \\
    &\ge& |R_\omega| - \tfrac 12 |T_i| - (k-1)(k-2).
  \end{eqnarray*}

  Assume now that $|R_\omega| < |B_\omega|$. For $j \neq i$, $S(T_j,a_j,b_j)$
  does not fail.  We have $\cor(\omega,T_j,a_j,b_j) = |R_\omega \cap T_j| >
  |B_\omega \cap T_j|$, if $|R_\omega \cap T_j| > b_j$, and
  $\cor(\omega,T_j,a_j,b_j) = |B_\omega \cap T_j|$, if $|R_\omega \cap T_j| \le
  a_j$.  If $|R_\omega \cap T_j| \in \{a_j+2, \ldots, b_j -1\}$, then
  \begin{eqnarray*}
    \cor(\omega,T_j,a_j,b_j) &=& \tfrac 12 |T_j| = |B_\omega \cap T_j| - (|B_\omega \cap T_j| -
  \tfrac 12 |T_j|)\\
  &\ge& |B_\omega \cap T_j| - (|T_j| - (a_j+2) - \tfrac 12 |T_j|) \\
  &\ge& |B_\omega \cap T_j| - (k-1).
  \end{eqnarray*}
  Hence $\cor(\omega,T_j,a_j,b_j)\ge |B_\omega \cap T_j| - (k-1)$ for all $j \neq
  i$. Together with $\cor(\omega,T_i,a_i,b_i) \ge |B_\omega \cap T_i| - \tfrac
  12|T_i|$, we conclude
  \begin{eqnarray*}
    \cor(S,\omega) &\ge& |B_\omega \cap T_i| - \tfrac 12 |T_i| + \sum_{j \in [k] \setminus
      \{i\}} \left(|B_\omega \cap T_j| - (k-1)\right) \\
    &\ge& |B_\omega| - \tfrac 12 |T_i| - (k-1)^2.
  \end{eqnarray*}
  This proves the claim.}
\end{proof}

\subsection{Optimizing the Partition}

It remains to choose a suitable partition $[n] = T_1 \dot\cup \ldots \dot\cup
T_k$. Let $k = \lc \sqrt[3]{n/4} \rc$. For any number $r \in \R$ denote by $\lc
r \rc_2$ the smallest even integer not smaller than $r$, and by $\lf r \rf_2$
the largest even integer not exceeding $r$.  Choose $\ell \in [k]$ such that $n
= \ell \lc n/k \rc_2 + (k-\ell) \lf n/k \rf_2$ --- recall that we assumed $n$ to be
even. Let $[n] = T_1 \dot\cup \ldots \dot\cup T_k$ be such that $|T_i| = \lc n/k
\rc_2$ for $i \in [\ell]$ and $|T_i| = \lf n/k \rf_2$ for $i \in [\ell
+ 1 \, .. \, k]$ and such that all $T_i$ respect our initially chosen pairing. Then
the loss compared to $\max\{|R_\omega|,|B_\omega|\}$ as given by the previous
lemma is at most
\begin{eqnarray*}
  \tfrac 12 \lc n/k \rc_2 + (k-1)^2 &\le&
1 + \tfrac 1 2 \sqrt[3]{4} n^{2/3} + \tfrac 1 {\sqrt[3]{16}} n^{2/3}
\le 1 + 1.2 n^{2/3}.
\end{eqnarray*}
This proves Theorem~\ref{thm}.

\end{document}